\documentclass{amsart}

\usepackage{amsmath, amssymb, graphicx, epsfig, verbatim, color}
\usepackage{epic}

\newtheorem{thm}{Theorem}[section]

\newtheorem{cor}[thm]{Corollary}
\newtheorem{lem}[thm]{Lemma}
\newtheorem{prop}[thm]{Proposition}

\newtheorem{defn}[thm]{Definition}
\newtheorem{conj}[thm]{Conjecture}
\newtheorem{quest}[thm]{Question}

\newtheorem{rem}[thm]{Remark}

\numberwithin{equation}{section}
\newcommand{\defin}[1]{{\bf\emph{#1}}}

\newcommand{\s}{\mathbf s}

\newcommand{\Z}{\mathbb Z}
\newcommand{\N}{\mathbb N}

\newcommand{\CP}{{\mathbb C}{\mathbb P}}

\newcommand{\hra}{\hookrightarrow}

\newcommand{\red}[1]{\textcolor{red}{#1}}

\begin{document}

\title{On the minimal genus problem in four-manifolds}

\author{Andr\'{a}s I. Stipsicz}
\address{R\'enyi Institute of Mathematics\\
H-1053 Budapest\\ 
Re\'altanoda utca 13--15, Hungary}
\email{stipsicz.andras@renyi.hu}

\author{Zolt\'an Szab\'o}
\address{Department of Mathematics\\
Princeton University\\
 Princeton, NJ, 08544}
\email{szabo@math.princeton.edu}

\begin{abstract}
  We study the minimal genus problem for some smooth
  four-manifolds.
\end{abstract}
\maketitle

\section{Introduction}
\label{sec:intro}
For a closed, oriented, smooth four-manifold $X$ the \emph{genus function} (or
minimal genus function) $g_X\colon H_2(X; \Z )\to \N$ assigns to each
homology class the minimal genus of an oriented smooth
surface representative of the class;
that is, if $\alpha \in H_2 (X;\Z )$ then
\[
g_X(\alpha )=\min \{ g(\Sigma )\mid i\colon \Sigma \hra X, 
\ i_*([\Sigma ])=\alpha \} , 
\]
where $\Sigma $ is an oriented, closed, connected surface of genus
$g(\Sigma )$ and $i$ is a $C^{\infty}$ embedding.  It is a standard
fact, that every second homology class in a four-manifold
can be represented by an
embedded, oriented, closed, connected surface, hence the above minimum
provides a finite value.

For some examples of exotic pairs of smooth four-manifolds
the genus functions show very different behavior, so it is expected
that this function codes important geometric information about the
smooth structure of the four-manifold.
In order to formalize this notion we have the following definition:

\begin{defn}
  Let  $X$ and $Y$ be smooth, closed, oriented
  four-manifolds with $H_1(X; \Z )=H_1(Y; \Z)=0$. We say that $X$ and $Y$ have
  \defin{equivalent minimal genus functions} if  there
  exists an isomorphism $f\colon
  H_2(X; \Z )\longrightarrow H_2(Y; \Z )$ with the properties that
  \begin{itemize}
\item $f$ preserves the intersection form:
  $$Q_X(a,b)= Q_Y(f(a),f(b))$$
  for $a,b\in H_2(X; \Z )$, and
  \item $f$ also preserves the minimal genus function: 
  $$ g_X(a) = g_Y(f(a))$$
    for all $a \in H_2(X; \Z )$.
    \end{itemize}
\end{defn}

A natural question regarding the above function is the following:

\begin{quest}\label{quest:kerdes}
  Let  $X$ and $Y$ be smooth, closed, oriented and  simply connected
  four-manifolds and assume that $X$ and $Y$ have equivalent minimal
  genus functions.
  Is there  an orientation preserving diffeomorphism between 
  $X$ and  $Y$?
\end{quest}

\begin{rem}
  Note that the above question only asks about the existence of an
  orientation preserving diffeomorphism $G\colon X \to Y$, and in
  particular does not require that the induced action $G_\ast \colon
  H_2(X)\to H_2(Y)$ agrees with a fixed $f$. Indeed, for $X=Y=K3$ the
  map $f\colon H_2(K3)\to H_2(K3)$ given by $f(a)= -a$ preserves both
  the intersection form and the minimal genus function, but (by work
  of Donaldson, see \cite[Corollary~9.1.4]{DK}), $f$ cannot be induced
  by a diffeomorphism.
\end{rem}

Despite its importance, very little is known in
general about $g_X$.
There are only a handful of simply connected
examples where the entire genus function is known. When $X=S^4$
(or more generally, $X$ is a homotopy $4$-sphere) the
genus function is trivial; in fact, the positive
answer for Question~\ref{quest:kerdes}
in this special case would imply the 
Smooth $4$-dimensional Poincare Conjecture (S4PC).

For $\CP ^2$ (and similarly, for
${\overline {\CP}}^2$) the resolution of the Thom conjecture by
Kronheimer and Mrowka \cite{KrMr} provides the answer: suppose that
$h\in H_2 (\CP ^2; \Z)\cong \Z$ is a generator and $d\neq 0$, then
\[
g_{\CP ^2}(dh)=\frac{1}{2}(\vert d \vert -1)(\vert d \vert -2).
\]
Along similar lines, Ruberman gave a description
of the minimal genus function for ${\CP}^2\#  {\overline {\CP}}^2$
and for $S^2\times S^2$ (the two $S^2$-bundles over $S^2$) in \cite{Ruber}:
if $(a,b)\in H_2({\CP}^2\#  {\overline {\CP}}^2; \Z )\cong \Z ^2$
(where the $\Z$-factors are generated by the projective lines in the factors)
is
a nontrivial homology class with $\vert a \vert >\vert b \vert$,
then
\[
g_{{\CP}^2\#  {\overline {\CP}}^2}(a,b)
=\frac{1}{2}(\vert a \vert -1)(\vert a \vert -2)-\frac{1}{2}\vert b\vert (\vert b \vert -1);
\]
if $\vert a\vert <\vert b \vert$ then the same formula applies
with the roles of $a$ and $b$ reversed, and
$g_{{\CP}^2\#  {\overline {\CP}}^2}(a,\pm a)=0$. In a similar manner, for
$(u,v)\in H_2 (S^2\times S^2; \Z )\cong \Z ^2$ (where the $\Z$ factors are
generated by the $S^2$-factors in $S^2\times S^2$) we have
\[
g_{S^2\times S^2} (u,v)=(\vert u \vert -1)(\vert v\vert -1)
  \]
  if $uv\neq 0$ and
  $g_{S^2\times S^2} (u,v)=0$ if $uv=0$.

Although these four-manifolds seem to be the simplest ones (definitely
with smallest Euler characteristics among simply connected closed
four-manifolds), it is not known if they support distinct smooth
structures.  Application of a suitable sequence of Luttinger surgeries
provide (not necessarily simply connected) smooth four-manifolds $X_n$
($n\in \N$) with homology isomorphic to $H_*(S^2\times S^2; \Z)$,
first constructed in \cite{FPS} by Fintushel, Park and Stern (see
also \cite{APS2xS2}). In addition, the construction can be carried out
so that $X_1$ carries a symplectic structure.  We will review the
construction of these manifolds in Section~\ref{sec:homs2s2} and prove
the following:

  \begin{thm}\label{thm:homologys2s2}
    Suppose that $X_n$ is one of the homology $S^2\times S^2$'s constructed in
    \cite{FPS} (see also Section~\ref{sec:homs2s2}). Using the isomorphism $H_2(X _n; \Z )=\Z^2$, where the two generators
      $x$ and $y$ satisfy $Q_{X_n}(x,x)=Q_{X_n}(y,y)=0$, $Q_{X_n}(x,y) =1$, 
    the minimal genus  on the non-zero homology class $\alpha =ax+by$
    is given by
\[
g_{X_n}(\alpha )=  (\vert a\vert +1)(\vert b\vert +1).
\]

In particular, for all $n,k\in \N$ the manifolds $X_n$ and $X_k$ have
equivalent minimal genus functions.
    \end{thm}

    Notice that (as Proposition~\ref{prop:SWcalc} will show), the
    Seiberg-Witten values do distinguish these four-manifolds; it is
    still an open problem whether this family contains simply
    connected examples.  It is also open if there is a pair $n,k\in
    \N$ so that $X_n$ and $X_k$ are homeomorphic --- it might happen
    that their fundamental groups distinguish them.

    Regarding Question~\ref{quest:kerdes}, as partial evidence for the
    opposite answer, we show that the minimal genus function does
    distinguish (even infinite collections of) exotic smooth
    four-manifolds.  For such example we will rely on two
    constructions on elliptic surfaces. First, recall that for given
    $n\geq 1$, $p, q >0$, $(p,q)=1$, the manifolds $E(n)_{p,q}$ are
    constructed from a simply connected genus $1$ Lefschetz fibration
    $E(n)$ over the sphere with Euler characteristics $\chi(E(n))=12n$
    using logarithmic transformations of multiplicity $p$ and $q$.
    (For more on elliptic surfaces see \cite[Chapter~3]{GS}.)
For $q=1$ the manifold $E(n)_{p,q}$ will be denoted by $E(n)_p$.
    
    A further construction of new four-manifolds out of old one is
    knot surgery \cite{FSknots}: suppose that $X$ contains an embedded
    torus $T$ with trivial normal bundle and with simply connected
    complement. For a knot $K\subset S^3$ replace the tubular
    neighbourhood $\nu (T)=T^2\times D^2$ of $T$ with $(S^3\setminus
    \nu (K))\times S^1$ (using an appropriate gluing map along the
    $T^3$ boundaries, where the meridian of the torus $T$ will be
    matched with the longitude of the knot $K \subset \partial
    ((S^3\setminus \nu (K)) \times S^1)$). We denote the result by
    $X_K$.  For an elliptic surface $E(n)$ we can choose $T$ to be the
    generic fiber of the fibration; the result of the above
    construction then will be denoted by $E(n)_K$.  We will prove the
    following
    
    \begin{thm}\label{thm:Elliptics}
For every  $n\geq 2$ there are infinitely many smooth four-manifolds
all homeomorphic to $E(n)$, distinguished 
by their genus function. Indeed, if $K_m$ denotes the $(2,2m+1)$ torus knot
$T_{2,2m+1}$, then $\{ E(n)_{K_m} \vert m\in \N\} $ have all distinct
genus functions.
    \end{thm}

    In the closing sections
    of the paper we investigate the minimal genus problem
    for some simply connected four-manifolds with small Euler
    characteristics. These include smooth four-manifolds that
    are homeomorphic to
    ${\CP}^2\# 2 {\overline {\CP}}^2$ or ${\CP}^2\# 3 {\overline {\CP}}^2$,
    constructed by Akhmedov and
    Park \cite{AP} and by Baldridge and Kirk \cite{BK}, respectively.
For surveys on the genus function, see also \cite{DorfLi, Lawson}.

\subsection{Tools used in the proofs}  
In determining the value of the genus function on a homology
class, the Symplectic Thom Conjecture (a theorem proved in \cite{OSzThom})
turns out
to be rather convenient.
\begin{thm} (Ozsv\'ath-Szab\'o, \cite{OSzThom})
    \label{thm:SymThomConj}
  Suppose that $(X,\omega )$ is a closed symplectic four-manifold and
  $\Sigma \subset X$ is a connected, closed symplectic surface of
  genus $g(\Sigma )$. Then $g_X ([\Sigma ] ) =g(\Sigma )$,
  i.e. symplectic surfaces minimize genus in their homology classes.
  \end{thm}

The following inequality --- called the \emph{adjunction
  inequality} --- does not determine the genus function
on a specific homology class, but provides bounds and
is applicable in a greater generality. Recall that the
Seiberg-Witten function of a smooth, closed, oriented, simply connected
four-manifold $X$ is a map $SW_X\colon {\rm Spin}^c(X )\to \Z$ which counts
solutions of a certain geometric partial differential equation.
(Here ${\rm Spin}^c(X)$ denoted the set of spin$^c$ structures on $X$.)
For manifolds with $b_2^+>1$ (or $b_2^+=1$ and $b_2^-\leq 9$ and with the
addition that the equations are considered with small perturbation terms)
this map is a smooth invariant. $SW_X$ has finite support, and
 $\s \in {\rm Spin}^c(X)$ with $SW_X(\s )\neq 0$ is called a
{\bf Seiberg-Witten basic class}. The four-manifold $X$ is of Seiberg-Witten
simple type if for a Seiberg-Witten basic class $\s$ 
we have that
$c_1^2(\s)=3\sigma (X)+2\chi (X)$ (where $\sigma (X)$ and
$\chi (X)$ are the signature
and the Euler characteristic of the underlying four-manifold).

\begin{thm}(Adjunction inequality, \cite[Corollary~1.7]{OSzThom})
  \label{thm:AdjIneq}
  Suppose that $X$ is a smooth four-manifold of Seiberg-Witten simple
  type with $b_1(X)=0$ and either with $b_2^+(X)>1$ or with
  $b_2^+(X)=1$ and $b_2^-(X)\leq 9$.  Then for an embedded surface
  $\Sigma \subset X$ with genus $g(\Sigma ) >0$ and a Seiberg-Witten
  basic class $\s\in {\rm Spin}^c(X)$ (i.e. with $SW_X(\s )\neq 0$) we
  have
\begin{equation}\label{eq:adjunction}
\vert \langle c_1(\s ), [\Sigma ]\rangle \vert +[\Sigma ] ^2\leq 2g(\Sigma )-2.
\end{equation}
\end{thm}

In our arguments we will frequently
use the fact that for a self-diffeomorphism $f\colon
  X\to X$ we have $g_X(\alpha )= g_X(f_*(\alpha ))$ for all $\alpha
  \in H_2(X;\Z )$. Furthermore, if $S\subset X$ is an embedded sphere
  of self-intersection $(\pm 1)$ or $(\pm 2)$ in $X$, then it is easy
  to construct a self-diffeomorphism $f\colon X\to X$, for which the
  induced map $f_*$ is just reflection in the class represented by
  $S$. Consequently, if $S\subset X$ is an embedded sphere with
  self-intersection $(\pm 1)$ or $(\pm 2)$, then the map
\[
\alpha \mapsto \alpha - 2\frac{[S]\cdot \alpha}{[S]\cdot [S]} [S]
\]
on $H_2 (X; \Z )$ is induced by a diffeomorphism.

The paper is organized as follows.  In Section~\ref{sec:homs2s2} we
examine the homology $S^2\times S^2$'s and verify
Theorem~\ref{thm:homologys2s2}.  In Section~\ref{sec:EllipticSurfaces}
we study the genus function of elliptic surfaces and of the manifolds
we get by logarithmic transformations or by knot surgeries on them. We
also prove Theorem~\ref{thm:Elliptics}.  Section~\ref{sec:CPk} is
devoted to the partial understanding of the genus function on the
twice and three-times blown up complex projective plane. The exotic
manifolds $AP$ and $BK$ (and the further infinite families of exotic
smooth structures on ${\CP}^2\# 2 {\overline {\CP}}^2$ and on
${\CP}^2\# 3 {\overline {\CP}}^2$), and their genus functions are
studied in Sections~\ref{sec:AP} and \ref{sec:BK}.

  {\bf Acknowledgements}: AS was partially supported by the
  \emph{\'Elvonal (Frontier) grant} KKP126683 of the NKFIH. He would
  also acknowledge the hospitality of MSRI/SLMath in Berkeley, where
  part of this project was carried out. The second author was
  partially supported by NSF grant DMS-1904628. We would like to thank
  the referee for their suggestions which significantly improved the
  presentation of the results in the paper.

  \section{Homology $S^2\times S^2$'s}
  \label{sec:homs2s2}
The application of torus (and in particular, Luttinger) surgery, or
some variation of it, leads to the construction of four-manifolds
which have the same homology as $S^2\times S^2$, but potentially
different fundamental group. Such examples were constructed in
\cite{FPS} (see also \cite{APS2xS2}).

First we briefly recall the concept of surgery along a torus in a
four-manifold from \cite{FPS}.  Assume that $X$ is a smooth
four-manifold and $T\subset X$ is an oriented torus with trivial
normal bundle (that is, $[T]^2=0$).  Fix a framing $f$ on $T$, i.e., a
trivialization of its normal bundle.  For an oriented simple closed
curve $\gamma \subset T$ therefore there is a unique push-off
${\overline \gamma}_f$ of $\gamma$ to the boundary $T^3$ of a small tubular
neighbourhood of $T$.  For a rational number $\frac{p}{q}$, the
manifold $(X, T, f, \gamma , \frac{p}{q})$ is defined by deleting a small
tubular neighbourhood of $T$, and gluing back $T^2\times D^2$ (which can
be viewed as the union of a four-dimensional 2-handle, two 3-handles
and a 4-handle) by attaching the 2-handle along a simple closed curve
representing the homology class $p[\mu _T] \pm q [{\overline \gamma }_f]$
(where $\mu _T$ is a meridional curve to $T$, regarded as a subset of
the boundary of the trivial neighbourhood, oriented so that it has
linking number 1 with $T$).

In the special case, when $(X, \omega )$ is a symplectic four-manifold
and $T\subset (X, \omega )$ is a Lagrangian torus (that is, $\omega
\vert _T=0$), the torus  $T$ comes with a canonical framing.
Indeed, by the Lagrangian neighbourhood
theorem an appropriate small normal neighbourhood of $T$ is
symplectomorphic to a neighbourhood of the zero-section of the
cotangent bundle $T^*T^2\to T^2$ of the 2-torus $T^2$ (when the total
space of this bundle is
equipped with the tautological symplectic form coming from its
Liouville 1-form). As $T^*T^2$ is trivialized, this identification
equips $T\subset X$ with its \emph{Lagrangian framing}.
(Pushing $T$ along this framing we get further Lagrangian
tori in $(X, \omega )$.) As is explained in \cite{ADK},
using this framing and applying
$\frac{1}{k}$  surgery along $T$ (for any choice of $\gamma$) the resulting
four-manifold $(X, T, \gamma , \frac{1}{k})$ will admit a symplectic structure
which is equal to $\omega$ outside of a suitably chosen neighbourhood of $T$.
This symplectic construction is frequently referred to as \emph{Luttinger 
  surgery} \cite{ADK, Lut}.

\subsection{The construction}
For the construction of the homology $S^2\times S^2$'s in
\cite[Section~4]{FPS}, consider two genus 2 surfaces $\Sigma ^1_2,
\Sigma _2^2$, and equip their first integral homologies with the
symplectic bases $a_1, b_1, a_2, b_2$ and $c_1, d_1, c_2, d_2$,
respectively.  Represent these homology classes by oriented simple
closed curves (denoted by the same letters). Equip the direct product
$\Sigma ^1 _2\times \Sigma ^2_2$ with the direct product symplectic
form.  Notice that in this symplectic four-manifold the product of an
oriented simple closed curve $\gamma _1$ in $\Sigma _2^1$ and $\gamma _2$ in
$\Sigma _2^2$ provides an oriented Lagrangian torus
$T_{\gamma _1,\gamma _2}\subset X$, hence using the chosen bases we see 16 tori
which are in 8 hyperbolic pairs. In addition,
by considering small disjoint displacements of the curves in $\Sigma ^i _2$,
tori with algebraic intersection
number 0 can be chosen to be disjoint. Apply now the following eight
Luttinger surgeries, all with the choice $k=\pm 1$: 
\begin{itemize}
\item $(T_{a_1,c_1}, a_1), (T_{b_1,c_1}, b_1), (T_{a_2,c_2}, a_2),
  (T_{b_2,c_2}, b_2)$,
\item $(T_{a_2,c_1}, c_1), (T_{a_2,d_1}, d_1), (T_{a_1,c_2}, c_2),
  (T_{a_1,d_2}, d_2)$;
  \end{itemize}
the first four with coefficient $k=-1$ and the last four with $k=1$,
resulting in the four-manifold $X$.
As each torus comes with a dual torus, the meridional curves for each
are homologically trivial. Therefore after the Luttinger surgery the
chosen curve will represent the trivial element in first homology,
hence $H_1(X; \Z )=0$.
By \cite[Proposition~2.2]{ADK} 
$X$ will admit a symplectic structure so that
the surfaces $\Sigma _2^1\times \{ p\}, \{q\} \times \Sigma _2^2$
are symplectic submanifolds.
The intersection form (given in the basis
$\Sigma _2^1\times \{ p\}, \{q\} \times \Sigma _2^2$) is a hyperbolic
pair, hence the manifold is spin and its cohomology ring is isomorphic to
$H^*(S^2\times S^2; \Z )$. By changing the last Luttinger surgery
(on $(T_{a_1,d_2}, d_2)$)  to a torus surgery with coefficient $n$ we get
a four-manifold $X_n$, which for $n>1$ does not carry a natural symplectic
structure.

We claim that the four-manifold $-X$ (the manifold $X$ constructed
above, with its orientation reversed) also admits a symplectic
structure. Indeed, consider the map $f\colon \Sigma _2^1\times \Sigma
_2^2\to \Sigma _2^1\times \Sigma _2^2$ which is given by the identity
on the first factor and by the orientation reversing diffeomorphism on
the second factor given by the reflection which leaves $c_1,c_2$
fixed, and sends $d_1$ and $d_2$ to $-d_1$ and $-d_2$.  The map $f$ is
an orientation reversing self-diffeomorphism of $\Sigma _2^1\times
\Sigma _2^2$, and all tori in $X$ we used in the Luttinger surgeries
above are invariant under $f$: the tori not involving $d_1$ and $d_2$
are fixed pointwise, while the tori $T_{a_2, d_1}$ and $T_{a_1,d_2}$
stay the same as subsets. Let $W$ denote the complement of the tubular
neighborhood of the eight tori in $X$ along which we perform Luttinger
surgeries, and note that the above map $f$ restricted to $W$ gives a
map form $W$ to $-W$. Let $\delta_i$ for $1\leq i\leq 8$ denote the
curves in the boundary of $W$ that define the Luttinger surgeries on
$X$.  Then using surgeries along $f(\delta_i)$ in $-W$ will give
$-X$. Since the map $f$ carries the meridians of the tori to meridians
(with potentially reversing their orientation) it follows that the
surgeries on $-W$ are still Luttinger surgeries, and $f$ extends
to those surgeries (as an orientation reversing map).  In particular $-X$
can be presented as a sequence of Luttinger surgeries along Lagrangian
tori in $\Sigma _2^1\times \Sigma _2^2$, therefore it is also
symplectic.

\begin{rem}
 Note that
 the above argument shows that the
 manifold $X$ admits a symplectic structure with both
 orientations, yet it is not diffeomorphic
  to any surface bundle.  Such examples were already known to exist.
\end{rem}

In the same manner, it follows that $-X_n$ is given from the product
of two genus-2 surfaces by a sequence of seven Luttinger surgeries and
a torus surgery with some integral coefficient, similarly to the
construction of $X_n$.

The property of admitting symplectic structure with both
orientation allows us to compute the genus function $g_X$ of $X$:
note that the Luttinger surgeries are
in the complement of the two genus 2 surfaces
$h_1=\Sigma _2^1\times \{ p\}$ and $h_2=\{ q\}\times \Sigma _2^2$
of $\Sigma _2^1\times \Sigma _2^2$ (for suitably chosen $p\in \Sigma _2^2$
and $q\in \Sigma _2^1$),
hence provide a basis of $H_2(X; \Z )$ with symplectic representatives.

\begin{proof}[Proof of Theorem~\ref{thm:homologys2s2} for $X=X_1$]
  The two generators $h_1, h_2$ can be represented by symplectic
  genus 2 surfaces. If for $\alpha =ah_1+bh_2$
  both $a$ and $b$ are positive, then by smoothing
  the intersection points of push-offs, we get a symplectic surface of the
  claimed genus. If both $a$ and $b$ are negative, we can take the negative
  of the homology class. If one coefficient is positive and the other is
  negative, then by reversing the orientation on the four-manifold we
  are in one of the previous cases. (Notice that here we use the fact that
  $-X$ is constructed in a similar fashion; in particular, it is symplectic.)
  Finally if one of the coefficients
  is zero, we represent the homology class near the
  other surface with an appropriate cover,
  and choosing it within a small neighbourhood, the cover can be assumed
  to be symplectic. Then simple Euler characteristic computation verifies
  the claimed formula.
    \end{proof}

By determining the Seiberg-Witten basic classes, the same result
can be shown to hold for all $X_n$.

\begin{prop}\label{prop:SWcalc}
  The smooth four-manifold $X_n$ admits two Seiberg-Witten basic classes,
  and so does $-X_n$ obtained by reversing the orientation on $X_n$.
\end{prop}
\begin{proof}
  The argument uses repeatedly a surgery formula for the Seiberg-Witten
  invariants, see \cite{MMSz, T}.

  For the version we need, we  fix a smooth closed oriented compact
  four-manifold $M$ with three-torus boundary $\partial M = T^3$, and a
  spin$^c$ structure {\bf s} whose first Chern class restricts trivially
  to $\partial M$. Fix two homology classes $a, b \in H_1(T^3)$ given
  by $a= [S^1\times pt\times pt]$ and $b = [pt\times S^1\times
    pt]$. Given a pair of integers $p,q$, where $(p,q) \in \Z\times
  \Z$ is a primitive element, let $M_{p,q}$ denote the effect of
  torus surgery, where we glue $M$ and $D^2\times T^2$ along $\phi$ to get
  $M_{p,q} = M\cup _\phi D^2\times T^2$, so that in homology $\phi
  _{\ast}$ maps $[\partial (D^2)]$ to $pa+qb$. Let $S(p,q)$ denote the set
  of spin$^c$ structures on $M_{p,q}$ whose restriction to $M$ agrees
  with ${\bf s}$ and define $F(p,q)$ as the sum of the Seiberg-Witten
  invariants of classes in $S(p,q)$. Then, according to \cite{MMSz}
\begin{equation}\label{eq:formula}
  F(p,q)= p\cdot F(1,0)+ q\cdot F(0,1).
\end{equation}
  The manifolds $X_n$ in the theorem are all constructed from
  $\Sigma_2\times \Sigma_2$ by using torus surgery eight times. Let
  $s_0$ denote the spin$^c$ structure on $\Sigma_2\times \Sigma_2$,
  with $c_1(s_0)$ evaluating as $2$ on both $[\Sigma_2\times pt]$ and
  $[pt \times \Sigma _2]$.  Note that the Seiberg-Witten invariant of
  $s_0$ is $\pm 1$.  Now we use Formula~\eqref{eq:formula}
  inductively. During these surgeries we use $a$ to be the meridional
  circle. So by choosing $b$ appropriately we can assume the effect of
  the torus surgery is given by $M(p,1)$, where $p=\pm 1$ in the first
  seven steps and $p=\pm n$ in the last step. The crucial observation
  is that in all these steps the manifolds $M(0,1)$ contain an
  embedded torus with self-intersection $0$ with the property that the
  first Chern class of $s$ evaluates as 2 on it. It follows from the
  adjunction inequality that $F(0,1) = 0$, so
  $$F(p,1) = p\cdot F(1,0).$$ Since all the $8$ surgeries are along
  tori that represent primitive second homology classes that are
  killed under the surgery, it follows that there is a unique spin$^c$
  structure in $S(p,1)$. This means that if we look at the nine
  different 4-manifolds $Y_i$, $i=0,..,8$ starting with
  $\Sigma_2\times \Sigma_2$ and ending in $X_n$, there is always a
  unique spin$^c$ structure $s_i$ that corresponds to $s_0$.  During the
  inductive step we see that $Y_i= M(1,0)$ and $Y_{i+1} = M(p,1)$. It
  follows that the Seiberg-Witten invariant of $s_8$ in $X_n$ is
  equal to $\pm n$.

  We proved that $s_8$ and $-s_8$ are two non-trivial Seiberg-Witten
  classes on $X_n$. Since the surgery tori are all disjoint from $pt
  \times \Sigma _2$ and $\Sigma_2\times pt$, it follows from the
  adjunction inequality that there are no other Seiberg-Witten basic
  classes.
  \end{proof}

\begin{proof}[Proof of Theorem~\ref{thm:homologys2s2} for $X_n$]
  The argument of the proof for $X=X_1$ applies almost
  verbatim: we consider the same surfaces constructed in that proof,
  and apply the adjunction inequality of Theorem~\ref{thm:AdjIneq}
  (providing lower bounds for values of the genus function) to verify
  that the constructions indeed provide surfaces minimizing genus.
  As the Seiberg-Witten calculations apply for $X_n$ and $-X_n$ as well,
  minimality of the genus of the constructed representatives
  follows from the respective adjunction formulae.
\end{proof}

\section{Elliptic surfaces}
\label{sec:EllipticSurfaces}
There are various ways to get invariants for an oriented, closed,
connected four-dimensional manifold $X$ from $g_X$ which are easier to
compare than the genus function itself. Below we consider
one such example. Let $Q_X$ denote the intersection
form of $X$ and $k=b_2^+(X)$.
\begin{defn}
  A $k$ element sequence ${\bf y} = (y_1,\ldots , y_k)$ with $y_i\in
  H_2(X, {\bf Z})$ is called an \defin{allowed sequence} if the span of
  $y_1,\ldots ,y_k$ is a positive definite $k$-dimensional subspace of
  $H_2(X, {\bf Z})$, and for all $i \neq j$ we have $Q_X(y_i,y_j)=0$
  and $Q(y_i,y_i)$ is even.
\end{defn}
  For an allowed sequence we associate a tuple of $k$ integers
  $$Pr(X,{\bf y}) =  (a_1, \ldots , a_k)$$
  where $a_i = g_X(y_i)$.

  \begin{defn}
  The \defin{genus profile} of $X$ is defined as the smallest among all the corresponding
  tuples of $k$ (with $k=b_2^+(X)$) integers in the lexicographical ordering:
  $$Pr(X) = min _{\bf y} Pr(X,{\bf y}).$$
  \end{defn}
  
  As an example, the genus profile of the $K3$ surface is $(2,2,2)$.
  A further example is given by $X = E(2)_{p}$ with $p$ odd,
  the $K3$ surface after a multiplicity $p$ logarithmic transformation
  has been applied.
  The minimal genus profile of $X$ will be $(2,2,f(p))$,
  since we can find two genus $2$ surfaces that are disjoint from each other
  and from the fiber of the fibration and have self-intersection $2$.
The function $f(p)$ can be estimated as follows:
  There is a  Seiberg-Witten basic class of the form $(p-1)\cdot z$ where $z\in H^2(X)$ has square zero and $z\neq 0$.
  Let ${\bf y}= (y_1,y_2,y_3)$ be an allowed sequence for
  which $$Pr(X)=Pr(X,{\bf y}).$$
  Using the adjunction inequalities we see that $z\cdot y_1=z\cdot y_2=0$,
  implying $z\cdot y_3 \neq 0$.
  Let $g$ denote the minimal genus of $y_3$. Using the the adjunction inequality we have 
  $$ 2g-2 \geq Q_X(y_3,y_3) + p-1 \geq p+1.$$
  It follows that the minimal genus is at least $(p+3)/2$.
  As $f(p)$ goes to infinity as $p\to \infty$, this argument already verifies
  that there are infinitely many $E(2)_p$'s with distinct genus functions.
  A similar, more explicit computation of the genus profile provides the
  proof of Theorem~\ref{thm:Elliptics}:
  
    \begin{proof}[Proof of Theorem~\ref{thm:Elliptics}]
      Consider $E(n)_{K_m}$, the result of knot surgery along the
      generic fiber of the elliptic fibration on $E(n)$ (a torus with
      trivial normal bundle), with $K_m=T_{2,2m+1}$ the alternating
      torus knot with crossing number $2m+1$ ($m\geq 1$). As the knot
      $K_m$ admits a Seifert surface of genus $m$, there is a surface
      in $E(n)_{K_m}$ of genus $m$ intersecting the generic fiber once
      and with self-intersection $-n$: we get this surface by gluing
      the section of $E(n)$ with the Seifert surface of $K_m$. By the
      computation of the Seiberg-Witten invariants of $E(n)_{K_m}$ in
      \cite{FSknots}, the adjunction inequality implies that this
      surface is of minimal genus in its homology class. The same holds
      for the surface we get by adding $k=\lceil \frac{n}{2}\rceil$
      copies of the fiber, so the result will have genus $m+k$ and
      positive self-intersection. This then shows that
      $Pr(E(n)_{K_m})= (2, \ldots , 2,k+m)$ (with $2n-2$ copies of
      $2$), verifying the result.
\end{proof}

\section{The genus function for ${\CP}^2\# 2 {\overline {\CP}}^2$ and
${\CP}^2\# 3 {\overline {\CP}}^2$}
\label{sec:CPk}

The genera of homology classes in ${\CP}^2\# 2 {\overline {\CP}}^2$
have been studied in \cite{Gao} and  \cite{LiLi}, see also \cite{Wall}.
Indeed, the genera of classes in
$H_2 ( \CP ^2\# 2 {\overline {\CP}}^2; \Z )$
with non-negative self-intersection
can be determined by a simple algorithm. To explain this algorithm,
we introduce the following
conventions. A homology class
$\alpha \in H_2( {\CP}^2\# 2 {\overline {\CP}}^2; \Z )$ will be
denoted by a triple $(a, b_1, b_2)$, reflecting the
decomposition
\[
H_2( {\CP}^2\# 2 {\overline {\CP}}^2; \Z ) =H_2({\CP}^2; \Z )
\oplus H_2( {\overline {\CP}}^2; \Z )\oplus 
H_2( {\overline {\CP}}^2; \Z )\cong \Z ^3.
\]

In $H_2( {\CP}^2\# 2 {\overline {\CP}}^2; \Z )$
the classes with one
coordinate $\pm 1$ and the others 0 can be obviously represented by
spheres. Therefore we can assume that in computing $g_{{\CP}^2\# 2
  {\overline {\CP}}^2}$, we have $a, b_1, b_2\geq 0$: otherwise we apply
a reflection to the coordinate in question (which is
generated by a self-diffeomorphism) to reverse its sign.
In the same way, by applying a diffeomorphism interchanging the two
${\overline {\CP}}^2$
components if necessary, we can assume
$b_1\geq b_2$.

{\bf{Genera of classes with $\alpha ^2 \geq 0$.}}
Suppose that $x=(a,b_1, b_2)$ has $x^2=a^2-b_1^2-b_2^2\geq 0$, and assume
that $a, b_1, b_2\geq 0$. If $a\geq b_1+b_2$, then the homology class
can be represented by a symplectic surface. Indeed, take $b_1$ lines in $\CP
^2$ passing through a given point $P_1$, take $b_2$ lines passing
through another point $P_2$ ($P_1\neq P_2$), and take $a-b_1-b_2$ further
generic lines, all disjoint from $P_1$ and $P_2$. Blowing up $P_1$ and $P_2$
we get a complex curve, which (after smoothing
the double points in a symplectic way) provides a smooth symplectic
surface representing the homology class in question.
By the Symplectic Thom
Conjecture (Theorem~\ref{thm:SymThomConj})
this minimizes genus in its homology class. The genus of this curve
is easy to determine: for such a triple $g(a,b_1,b_2)$ is equal to 
\begin{equation}\label{eq:GenusFct}
\frac{1}{2}((a-1)(a-2) -b_1(b_1-2)-b_2(b_2-1)).
\end{equation}

The following idea, determining the genus function for all classes
with non-negative self-intersection, already appeared in \cite{Gao, Wall}:
Consider a class $\alpha =(a,b_1,b_2)$ and let $c$ denote
$a-b_1-b_2$. For $c\geq 0$ the genus of $\alpha$ has been described
above. Suppose now that $c<0$.  Consider the sphere representing the
homology class $(1,1,1)$ (which has self-intersection $(-1)$) we get
by connecting the projective lines in the projective plane factors,
and apply a diffeomorphism $f$ inducing  reflection to this sphere. The
induced map $f_*$ takes $(a,b_1,b_2)$ to $(a+2c, b_1+2c, b_2+2c)$. By
applying further diffeomorphisms we can assume that these integers are
non-negative, and (as $c<0$) it is easy to see that $\vert a+2c\vert
<a$. Therefore after finitely many steps of the above algorithm we get
a vector $\alpha '=(a', b_1', b_2')$ with
$c'=a'-b_1'-b_2'\geq 0$ and with $g(\alpha )=g(\alpha ')$.
For $\alpha '$ the formula of Equation~\eqref{eq:GenusFct} above
applies and determines the value of the genus function.

If $\alpha ^2=0$ (i.e.,
when $a^2-b_1^2-b_2^2=0$, that is, $(a,b_1,b_2)$ is a Pythagorean
triple), the assumption $a^2-b_1^2-b_2^2=0$ together
with $a-b_1-b_2\geq 0$ implies that one of $b_1$ or $b_2$ is equal to
zero, hence the above algorithm stops at a vector of the form
$(a,a,0)$ (or $(a,0,a)$). These classes can be represented by complex
curves which consist of $a$ copies of $\CP ^1$. By smoothly connect
summing them, we get
\begin{cor}
  If $\alpha \in H_2 ({\CP}^2\# 2 {\overline {\CP}}^2; \Z )$ has
  $\alpha ^2=0$ then $\alpha$ can be represented by a sphere
  (i.e. $g(\alpha )=0$).
  \end{cor}

{\bf{Classes with $\alpha ^2<0$}.}
For classes of negative self-intersection we do not have a complete
picture. There are a number of cases when one can find a sphere representing
the class, while in some cases we have a lower bound for the genus.
This latter case emerges mainly for characteristic elements.

A similar algorithm as before applies for classes of negative square
and (after applying diffeomorphisms) we need to examine
only those triples $(a, b_1, b_2)$ which satisfy $2a\geq b_1+b_2$.
Indeed, by  taking
the sphere of self-intersection 2 representing the class $(2,1,1)$,
we can apply reflection to a class with $d=2a-b_1-b_2$ negative (and
further diffeomorphisms to regain $a,b_1,b_2\geq 0$), and this process
will stop at a representative of the diffeomorphism orbit
satisfying $2a\geq b_1+b_2$. This information is, however, not sufficient
to determine the genus function on the homology class.
Indeed, in the above form (as $b_1\geq b_2$) we can assume that $a\geq b_2$, so
the second blow-up can be performed in a point of $\CP ^2$ where
$b_2$ lines pass through, but for the rest of the surface
representing the class $(a,b_1, b_2)$ we need to apply ad hoc methods, which
sometimes give better results for those representatives of the diffeomorphism
orbit
of the class which fail to satisfy $2a\geq b_1+b_2$.

There are simple examples of homology classes which can be represented
by an embedded sphere; for example:

\begin{lem}\label{lem:ExamplesOfSpheres}
  Consider the homology classes $(a,a+1,c)$ with 
  $c=1,2$.
  These classes can be represented by embedded spheres.
\end{lem}
\begin{proof}
   The class $(a,a+1)\in
  H_2(\CP ^2\# {\overline {\CP}}^2; \Z )$ can be represented by a
  sphere. Connect summing with the sphere representing $c=1,2$ we find
  the sphere representative of the class.
  \end{proof}

Lower bounds for the genus function for homology classes with
negative self-intersection are scarce; there are
some bounds for 
 characteristic homology classes, though, which we present below.
 Note that the class
$\alpha =(a,b_1,b_2) \in H_2( {\CP}^2\# 2 {\overline {\CP}}^2; \Z )$
is characteristic if and only if the coordinates $a,b_1, b_2$ are all odd.
Suppose now that the characteristic homology class
$\alpha$ can be represented by an embedded sphere.
Consider the $(\vert \alpha ^2\vert -1)$-fold antiblow-up
(i.e. connected sum with $\CP ^2$) of ${\CP}^2\# 2 {\overline {\CP}}^2$,
where the (still characteristic) homology
class $(a,b_1, b_2, 1, \ldots, 1)$ can be now represented
by an embedded sphere of self-intersection $(-1)$. We can blow that
sphere down, and as the class was characteristic, its complement (and hence
the result of the blow-down) is a spin four-manifold. Applying Furuta's
$\frac{10}{8}$-theorem \cite{Furuta} giving
$b_2(X)\geq \frac{5}{4}\vert \sigma (X)\vert +2$ and
the divisibility of $\sigma (X)$ by 16
to this spin four-manifold, we get

\begin{lem}
  Suppose that $\alpha =(a,b_1,b_2) \in H_2(
  {\CP}^2\# 2 {\overline {\CP}}^2; \Z )$
  is a characteristic homology element with negative self-intersection
  $\alpha ^2=-n$ ($n\in \N$).
  If $\alpha $ can be represented by an embedded sphere then $n$ is 
  $1$. \qed
  \end{lem}

A more general lower bound of the genus function on characteristic elements
can be derived as follows. Suppose that $\alpha $ is a characteristic element
in $H_2({\CP}^2\# 2 {\overline {\CP}}^2; \Z )$
with $\alpha ^2=-8n-1$. (Recall that the square of a characteristic element
is congruent to the signature of the ambient four-manifold mod 8.)
We claim that
\begin{prop}
  The smooth representative of a characteristic element $\alpha
  \in  H_2 ({\CP}^2\# 2 {\overline {\CP}}^2; \Z )$ with
  $\alpha ^2=-8n-1$ has genus at least $\frac{n}{2}$.
\end{prop}
\begin{proof}
  As the characteristic element represents the Poincare dual of
the second Stiefel-Whitney class $w_2$ of the ambient manifold
${\CP}^2\# 2 {\overline {\CP}}^2$, 
the complement of the normal disk bundle
admits a spin structure.

Suppose now that the smooth representative of $\alpha $ is of genus
$g<\frac{n}{2}$. Consider the connected sum of $8n+1$ copies of
$\CP^2$ when $n$ is even, and $8n-7$ copies if $n$ is odd.  Below we
concentrate on the case of even $n$: choose in every summand the
projective line, connect sum them, and add $g$ fake handles, resulting
in a surface $S$ in $\#_{8n+1}\CP ^2$
of genus $g$ and self-intersection $8n+1$,
which represents a characteristic class (hence its complement is spin).

In constructing the normal connected sum of the two four-manifolds
along these surfaces we can choose the gluing map so that the result
is a spin manifold $Z$.  Indeed, the spin structures on a genus $g$
surface come in two types depending on their Arf invariants (see the
discussion in \cite[page~103]{kirby}), and two spin structures are
diffeomorphic iff they have equal Arf invariants. The spin three-manifold
$Y$ we get as the boundary of the complement of the open neighborhood of
the characteristic surface is a circle bundle over a surface (of odd
Euler number). Deleting a small tubular neighbourhood of a fiber,
we get a trivial bundle over a surface with boundary. On this
surface the exact
same classification of spin structures holds as for the closed case,
implying that on the circle bundle there are two orbits of spin structures
under the diffeomorphism group. The two structures can be given either as
above, or considering the surface in the $(8n-7)$-fold connected sum
of $\CP ^2$, but choosing the cubic curve in one of the components as opposed
to the projective line. As gluing with this latter spin structure we would
get a four-manifold violating Rokhlin's theorem about signatures of
closed spin manifolds, it follows that the two spin structures on
the boundaries in our construction are diffeomorphic, hence
(by choosing the right gluing map) provide a
spin structure on $Z$.

As $S$ has simply connected complement,
the resulting manifold will have $b_1(Z)=0$, hence its signature
$\sigma (Z)=8n$ and its second Betti number $b_2(Z)=\chi (Z)-2=8n+4g+2$
satisfies the $\frac{10}{8}$-theorem of Furuta \cite{Furuta}, implying
\[
b_2(Z)\geq \frac{5}{4}\sigma (Z) +2,
\]
eventually providing $g\geq \frac{n}{2}$. (An almost identical argument
gives the bound when $n$ is odd.)
\end{proof}

An upper bound for the genus can be given by the following simple
algorithm. First transform the given homology class to $(a, b_1 , b_2)$
with all nonnegative and $b_1\geq b_2$, and in addition $2a\geq b_1+b_2$.
In particular, we can assume that $a \geq b_2$.

As $a-b_2<b_1$ (which follows from $a^2-b_1^2-b_2^2<0$) and $a\geq
b_2$ for $\alpha =(a,b_1,b_2)$ with $\alpha ^2<0$ and $2a\geq
b_1+b_2$, we can consider $b_2$ lines passing through a given point
$P_1$ and the remaining $a-b_2$ through another point $P_2$.  Blowing
up the two points, we get a symplectic representative of the class
$(a, a-b_2, b_2)$. Addig $b_1+b_2-a$ copies of the first exceptional
divisor to it, we get the right homology class, while by resolving the
negative intersection points the result will not necessarily be
symplectic anymore. The genus of the resulting surface gives an upper
bound for the minimal genus $g(\alpha )$.

\subsection{The case of ${\CP}^2\# 3 {\overline {\CP}}^2$}
Most of the above arguments provide (after some minor modifications)
results for the genus function of
${\CP}^2\# 3 {\overline {\CP}}^2$.
As before, for a class $\alpha =(a, b_1, b_2, b_3)
\in  H_2 ({\CP}^2\# 3 {\overline {\CP}}^2; \Z )$ with
$\alpha ^2\geq 0$
there is a diffeomorphism under which the coordinates
of the image satisfy $a\geq b_1+b_2+b_3$ (and all non-negative).
In this step we use the $(-2)$-sphere representing the
homology class $(1,1,1,1)\in H_2 ({\CP}^2\# 3 {\overline {\CP}}^2; \Z )$.
Under the inequality $a\geq b_1+b_2+b_3$
the homology class can be represented by a symplectic
submanifold: as before, consider three distinct points
$P_1,P_2,P_3$ in $\CP ^2$, $b_i$ lines through $P_i$ (and
$a-b_1-b_2-b_3$ in generic position, avoiding all the $P_i$)
and blow up $P_i$. After symplectically
smoothing the double points of the resulting
complex curve, we get the desired symplectic submanifold.
This procedure again gives an algorithm, rather than a
formula to determine the genus function
$g_{{\CP}^2\# 3 {\overline {\CP}}^2}$.
Once again, if $\alpha =(a,b_1, b_2, b_3)$ satisfies $b_1\geq b_2\geq b_3$
(which can be assumed without loss of generality) and $\alpha  ^2=0$, then
after some diffeomorphism $\alpha =( a,  a, 0,0)$, and this class
can be represented by a smoothly embedded sphere.
Hence
\begin{cor}
  All classes $\alpha \in H_2 (H_2 ({\CP}^2\# 3 {\overline {\CP}}^2; \Z )$
  with $\alpha ^2=0$ have
  $g_{{\CP}^2\# 3 {\overline {\CP}}^2}(\alpha )=0$.
  \end{cor}


As in the case of $\CP ^2 \# 2 {\overline {\CP}}^2$, we were not able
to determine the genus function on all elements with negative
self-intersection in $H_2 ({\CP}^2\# 3 {\overline {\CP}}^2; \Z )$.
The argument for finding a bound earlier for the genera of characteristic
elements applies in this case as well, providing some constraint, but
in general such results are not sufficient for getting exact values of
the genus function.

\section{The Akhmedov-Park exotic ${\CP}^2\# 2 {\overline {\CP}}^2$ manifold $AP$}
\label{sec:AP}
In \cite{AP} an infinite family of pairwise non-diffeomorphic smooth
four-manifolds has been constructed, each homeomorphic to ${\CP}^2\# 2
{\overline {\CP}}^2$.  One of these examples, which will be denoted by
$AP$, was shown to carry a symplectic structure; we recall the
definition of these manifolds below. The goal of this section is to
partially determine the genus function of these exotic manifolds.

The study of $AP$ (and indeed, all the exotic
copies discussed below) is different from ${\CP}^2\# 2 {\overline {\CP}}^2$
in two crucial aspects. First, we do not see the spheres which then
provide the diffeomorphisms inducing reflections in homology. (Indeed,
we suspect that $AP$ contains no homologically nontrivial sphere at
all, cf. Conjecture~\ref{conj:NoSphere}.)
The other aspect, however, provides more tools to study the
genus function: the proof of exoticness shows that the
symplectic form $[\omega ]$ and its first Chern class $c_1(AP, \omega )$
pair negatively, hence $AP$ has Kodaira dimension $\kappa =2$ (a
'symplectic four-manifold of general type'). This implies that
the adjunction inequality holds true for $AP$:
As this theorem holds for negative self-intersections as well, we find more
restrictions in $AP$ for classes of negative square.
In a similar manner, the further exotic smooth manifolds admit
nontrivial Seiberg-Witten invariants, hence the adjunction
inequality also provides bounds on the genera of homology classes (no matter
what square they have).

\subsection{The construction of the manifold $AP$}

The symplectic exotic example is given as a symplectic normal
connected sum of two symplectic manifolds along symplectic genus 2
surfaces, followed by a sequence of appropriately chosen Luttinger
surgeries.  Indeed, consider $S_0^3(T)$, the result of 0-surgery along
the trefoil knot $T\subset S^3$, and let $X_1=S^3_0(T)\times
S^1$. This is a symplectic four-manifold, as it admits a
$T^2$-fibration structure over $T^2$ (with a section) --- resting on the
fact that $T$ is a fibered knot of genus 1.
Let $t_1$ be an embedded torus representing the fiber and $t_2$
representing a section of this bundle. Indeed, the homology class
$2[t_1]$ also admits a torus representative (by braiding the
generic fiber), which we denote by
$t_1'$, and all of these tori can be chosen to be symplectic. As $t_1'$
and $t_2$ intersect in two points, by blowing up one of them and resolving
the
other, we get a genus 2 symplectic surface $\Sigma\subset X$
with self-intersection 0 in $X=(S^3_0(T)\times S^1) \# {\overline
  {\CP}}^2$.  Take the fiber connected sum of $(X, \Sigma )$ with
$(\Sigma _2\times T^2 , \Sigma _2\times \{p\} )$ for the genus 2
surface $\Sigma _2$ and for some point $p$ in the 2-torus $T^2$.  The
resulting manifold $Y$ will be a symplectic manifold. We can identify
a few interesting submanifolds in $Y$: for example, the symplectic
genus 2 surface $\Sigma $ along which the normal sum was taken
(after pushed off) provides a symplectic
genus 2 surface in $Y$ with self-intersection 0, which we will denote
by $A$. In a similar manner, two further surfaces $B$ and $C$ can be
constructed as follows: take a torus in $\Sigma _2\times T^2$
intersecting the genus 2 surface we will use in the symplectic normal
sum transversely in a single point, and glue to a push-off of either
$t_1$ or $t_2$. (Notice that $t_1$ intersects $\Sigma$ once, so after
the gluing we get a genus 2 surface $B$, while $t_2$ intersects
$\Sigma$ in two points, so in the construction we need to use two
copies of the torus in $\Sigma _2\times T^2$, and the resulting
surface $C$ will have genus 3.)  A fourth surface $D$ can be
constructed by taking the exceptional divisor of the blow-up in
$(S^3_0(T)\times S^1) \# {\overline {\CP}}^2$. This sphere is
punctured twice in the symplectic normal sum procedure, and gluing it to 
punctured tori on both sides we get a genus 2 surface $D$.
All the above surfaces can be constructed to be
symplectic.

The main result of Akhmedov and Park in \cite{AP} is that by applying
Luttinger surgeries in $Y$ we can get a simply connected four-manifold.
(As Luttinger surgery does not change signature and Euler characteristic,
the result will be necessarily homeomorphic to the two-fold blow-up
of $\CP ^2$.)

\begin{thm} (\cite{AP})
  There are four tori in $Y$ (originating from $\Sigma \times T^2$) along which
  Luttinger surgeries (with appropriately chosen coefficients)
  provide a simply connected four-manifold, and the 
 symplectic structure of $Y$ extends to the resulting
  four-manifold $AP$.
\end{thm}
\begin{rem}
  In \cite{AP} the coefficients are carefully described, and the fundamental
  group calculation is given in details.
\end{rem}

The Luttinger surgeries performed in $\Sigma _2\times T^2$ are disjoint
from the surfaces introduced above in $Y$, hence $A,B,C,D$ give rise to
surfaces (also denoted by the same letters) in $AP$. It is not hard to see that
if $c_1$ denotes the first Chern class of the symplectic
structure on $AP$, then
  
\begin{itemize}
\item $g(A) =g(B)=g(D)=2$, $g(C)=3$.
\item $c_1(A)=c_1(B)=2$, $c_1(C)=4$, $c_1(D)=3$.
\item $A\cdot A = B\cdot B = C\cdot C = 0$, $D\cdot D = -1$.  
\item $ B\cdot D = C\cdot D = 0$
\item $A\cdot B = 1$, $A\cdot C = 2$, $B\cdot C = 1$, $A\cdot D =2$.
\end{itemize}

Indeed, $H_2(AP; \Z )$ is generated by $B,C,D$, and we have a relation
$A = 2B +C -2D$. As these surfaces are symplectic,
Theorem~\ref{thm:SymThomConj} implies that the above
genera are indeed minimal. This observation then shows

\begin{thm}\cite{AP}\label{thm:APExotic}
  The four-manifold $AP$ is not diffeomorphic to 
  ${\CP}^2\# 2 {\overline {\CP}}^2$.
\end{thm}
\begin{proof}
  The genus function is a diffeomorphism invariant, and it vanishes on
  the subset of $H_2({\CP}^2\# 2 {\overline {\CP}}^2; \Z )$ consisting
  of elements of zero square, while in $AP$ there are symplectic
  submanifolds of square zero and nonzero genus (any of $A,B,C$ above
  will do).
\end{proof}

Indeed, we can determine  the function $g_{AP}$ for a large class of
elements. Let $\alpha  = t_A\cdot A +t_B\cdot B +t_C \cdot C$, where
all $t_i \geq 0$ (and at least one of them is non-zero).
This class can be represented by a symplectic surface,
and its genus can be determined by a simple (although not very
enlightening) formula. Indeed, the multiple of a class can be represented
by a cover of the 0-section in its (trivial) normal bundle, which can
be chosen to be symplectic. 
The sum of the terms then can be represented by the union of these
potential covers, and the positive intersections can be symplectically
resolved.

The lower bound provided by Equation~\eqref{eq:adjunction}
of the adjunction inequality 
cannot provide sharp results in general --- indeed, when the
self-intersection of a homology class is negative, then for
sufficiently large multiples of the class the inequality of
Equation~\eqref{eq:adjunction} does not provide any information.
For this reason, at the moment we do not have effective tools to
estimate the genus function for homology classes in general, and
this implies that $g_X$ cannot be determined in general.

Notice that the generic homology class in $H_2({\CP}^2\# 2 {\overline
  {\CP}}^2; \Z )$ can be given as $t_BB+t_CC+t_DD$, but since $D$ has
self-intersection $-1$, even if all coefficients are non-negative, the
direct argument given above does not necessarily provide a symplectic
submanifold.  The same phenomenon can happen if the symplectic surface
representing a given homology class happens to be disconnected: in
this case smoothly we can tube the components together, but this
cannot necessarily be done symplectically.  In the following we show
some cases where a better understanding of the construction of $AP$
and a surface representing a particular homology class provides
results close to optimal (and sometimes even optimal).

{\bf Some examples:}
\begin{itemize}
\item Consider $\alpha =B+C+D$. As $B,C,D$ all can be represented
  by a symplectic surface, we have a disconnected representative of $\alpha$
  (with components of genus 5 and 2); tubing them we get a surface
  of genus 7 representing $\alpha$. As $c_1(\alpha )=9$,
  the adjunction formula of Theorem~\ref{thm:AdjIneq}
  for any smooth
  representative gives $g\geq 6$ as a lower bound.
 \item For $\alpha =B+C-D$ we have $c_1(\alpha )=3$ and $\alpha ^2=1$,
  so adjunction implies $g\geq 3$. A surface representing this class
  can be given by taking the disconnected symplectic representative
  of $B+C+D$, reverse the orientation on the component corresponding
  to $D$ and then tube. The result will be a surface of genus 7.
  A better result can be achieved by the following simplification:
  when we construct the individual surfaces $B$ and $D$, in the
  $\Sigma _2\times T^2$-part we consider two punctured genus-1
  surfaces (once with coefficient 1 in $B$ and once with coefficient $-1$ in
  $D$). Substituting these subsurfaces with an annulus we get a genus 5
  representative of the class $\alpha$.

\item Let $\alpha = B+2C+D$; then $\alpha ^2 = 3$ and $c_1(\alpha )=
  13$. The lower bound for the genus provided by the adjunction
  formula is 9. A genus 10 surface can be easily constructed by taking
  the disconnected symplectic representative (with a genus-8 and a
  genus 2 components) and tube them.
\item For the class $\alpha =2B+2C-2D$ we have
  $\alpha ^2=4$ and $c_1(\alpha )=6$, so the adjunction
  inequality provides 6 as the minimal genus.
  Noticing that $A=2B+C-2D$ is a symplectic surface of genus 2 intersecting
  $C$ in two points, after resolving the intersection points, we find
  a genus 6 symplectic surface representing the homology class at hand.
\item Let us take $\alpha = B+D$. In this case $\alpha ^2=-1$ and
  $c_1(\alpha ) =5$, so the adjunction inequality provides 3 as a lower bound.
  The disconnected representative (after tubing) provides a genus 4 surface.
  Now if we use (instead of the two punctured tori,  each appearing
  in $B$ and $D$ on the $\Sigma _2\times T^2$-side) a twice-braided torus, which
  now intersects the normal connect sum surface, we get a representative of
  genus 3.
  

\end{itemize}

After an extensive search we did not find a nontrivial homology class in
$H_2(AP; \Z )$ which we could represent by a sphere. For classes of
self-intersection at least $-1$ there is no such class, as can be
deduced from Seiberg-Witten theory, but for more negative squares this
is far from obvious.  (Compare with the examples of
Lemma~\ref{lem:ExamplesOfSpheres} for $\CP ^2\# 2{\overline
  {\CP}}^2$.) Indeed, the search for embedded tori (representing
non-trivial homology classes) did not provide any result either.
This led us to the following
\begin{conj}\label{conj:NoSphere}
  The four-manifold $AP$ does not contain homologically essential embedded
  spheres or tori.
  \end{conj}

%
%

As it happened in the case of
the homology $S^2\times S^2$ examples, the last Luttinger surgery in
the construction can be replaced by a torus surgery with coefficient
$n$, providing an infinite family $Z_n$ of four-manifolds, distinguished
by the value of the Seiberg-Witten function on the two basic classes.
(See the computation in Proposition~\ref{prop:SWcalc}; above
$Z_1=AP$.)
As the above constructions all happened in the complement of
all surgeries, the constructions are still valid, and whenever
the minimality is verified by the adjunction inequality, the statement holds
true for the entire family. (Notice that these manifolds --- with the exception
of $AP$ --- are not symplectic.)
In conclusion, whenever we were able to identify the value of the
genus function we get $g_{Z_n}(\alpha )=g_{AP}(\alpha )$, leading us to
\begin{conj}
  The genus functions of the exotic four-manifolds $Z_n$ are all
  equivalent to $g_{AP}$.
\end{conj}

\begin{rem}
  The Seiberg-Witten functions $SW_{Z_n}\colon {\rm Spin}^c(Z_n)\to \Z$
  all have the same support (the two basic classes of $Z_n$), but
  different values on their support. As the adjunction formula (and
  any genus-related formula) is insensitive for the Seiberg-Witten
  {\bf{value}}, there is reason to believe that the genus functions
  for $Z_n$ are all the same. At the same time, this observation raises
  the question how the Seiberg-Witten value influences the geometry
  of the underlying four-manifold.
\end{rem}

Conjecture~\ref{conj:NoSphere} naturally
extends to the four-manifolds $Z_n$: we expect that these
four-manifolds do not contain homologically essential embedded
spheres or tori.

\section{The Baldridge-Kirk exotic ${\CP}^2\# 3 {\overline {\CP}}^2$ manifold $BK$}
\label{sec:BK}
A similar construction as in the previous section provides exotic
structures on ${\CP}^2\# 3 {\overline {\CP}}^2$, as it was
demonstrated in \cite{BK} by Baldridge and Kirk. Indeed, the building
blocks are very similar to the ones used earlier: consider the
four-manifold $S^3_0(T)\times S^1$ for the trefoil knot $T$, which
admits a $T^2$-fibration over $T^2$ with a section. Resolve the union
of a section and a fiber and blow up the resulting genus 2 surface
twice, getting a genus 2 surface $\Sigma $ with trivial normal bundle
in $(S^3_0(T)\times S^1)\# 2 {\overline {\CP}}^2$.  Consider the
symplectic normal sum of the resulting symplectic surface in the
symplectic four-manifold with $\Sigma _2\times T^2$ along $\Sigma
_2\times \{ p\}$ for some $p\in T^2$. 
\begin{thm} (\cite{BK})
  There are four tori in the above normal connected sum (originating
  from $\Sigma \times T^2$) along which Luttinger surgeries (with
  appropriately chosen coefficients) provide a simply connected,
  symplectic four-manifold, which we call $BK$.
\end{thm}

Using the pieces, and considering surfaces before the Luttinger
surgeries, one can find several symplectic submanifolds in $BK$.
First, there is the genus 2 surface used in the symplectic normal sum,
which provides a symplectic submanifold with self-intersection 0; the
corresponding homology class will be denoted by $A$. There is a
punctured symplectic torus in $\Sigma _2 \times (T^2\setminus \{p\})$,
and as it can be chosen to be disjoint from the Lagrangian tori of the
Luttinger surgeries, it descends to $BK$. From the other side, the
section and the fiber of the $T^2$-fibration on $S_0^3(T)\times S^1$
give punctured tori in the complement of the genus 2 surface; when
gluing them with the punctured torus in $\Sigma _2\times (T^2\setminus
\{p\})$, we get two surfaces $B$ and $C$ of genus 2, self-intersection
0, intersecting each other and $A$ once.  The exceptional divisors of
the blow-ups on $S_0^3(T)\times S^1$ provide two punctured disks with
relative framing $(-1)$, hence after gluing them to the punctured tori
in $\Sigma _2\times (T^2\setminus \{p\})$ we get two disjoint tori of
self-intersection $(-1)$, and call them $D_1$ and $D_2$.

A basis of $H_2(BK; \Z )\cong \Z ^4$ can be given by $B,C,D_1, D_2$;
the genera and intersection patterns are described above, and the first
Chern class $c_1$ of the symplectic four-manifold $BK$ evaluates as
$c_1(B)=c_1(C)=2, c_1(D_1)=c_1(D_2)=1$, while $c_1(A)$ is also 2.

Obviously, the genus function $g_{BK}$ is nonzero on the subset of
$H_2(BK; \Z )$ formed by elements of self-intersection 0 (as
$g_{BK}$ takes the value 2 on $A,B,C$), verifying again that the
diffeomorphism invariant $g_X$ distinguishes
${\CP}^2\# 3 {\overline   {\CP}}^2$
from $BK$, the straightforward extension of
Theorem~\ref{thm:APExotic} for the manifold $BK$. (Notice that in the
partial determination of $g_{BK}$ we rely on the
Symplectic Thom Conjecture Theorem~\ref{thm:SymThomConj}.)

Once again, homology classes of the form $t_A\cdot A + t_B \cdot
B+t_C\cdot C$ with $t_A, t_B, t_C\geq 0$ can be represented by
symplectic submanifolds, and hence the genus function on those classes
can be algorithmically determined, while in general the adjunction
inequality of Theorem~\ref{thm:AdjIneq} provides a lower bound on the
genus (in terms of the value of $c_1$ on the homology class and its
self-intersection), while specific constructions give upper bounds on
$g_{BK}$. The two bounds, however, do not meet in general.

As before, the last Luttinger surgery in the construction of $BK$ can
be replaced by an appropriate sequence of torus surgeries, providing a
sequence of four-manifolds $V_n$ (with $V_1=BK$), which then show the
existence of infinitely many distinct smooth structures on ${\CP}^2\#
3 {\overline {\CP}}^2$ (see \cite{BK}). As in the previous section,
these manifolds are distinguished by the Seiberg-Witten values on the
basic classes.  We cannot compare their genus functions, as the values
of these functions on most homology classes of negative
self-intersection are unknown at the moment.  
We do expect, however
\begin{conj}
  The exotic four-manifolds $V_n$ have equivalent genus functions.
\end{conj}


In contrary to $AP$, in the four-manifold $BK$ it is easy to find
tori and spheres. Indeed, the homology classes $D_1, D_2$ admit
toric representatives. In addition, when applying the two
blow-ups in $S_0^3(T)\times S^1$, we can arrange the blow-ups so that
the resulting four-manifold has a $(-2)$-sphere disjoint from
the genus 2 surface along which the normal sum is taken. Therefore
this sphere embeds into $BK$ (and indeed to all $V_n$'s).


\begin{thebibliography}{1}

\bibitem{AP}
  A. Akhmedov and D. Park,
  \newblock {\it Exotic smooth structures on small
    4-manifolds with odd signatures,}
  \newblock Invent. Math. {\bf{181}} (2010),  577--603.

  \bibitem{APS2xS2}
    A. Akhmedov and D. Park,
    \newblock {\it Exotic smooth structures on $S^2\times S^2$,}
    arXiv:1005.3346

  \bibitem{ADK}
    D. Auroux, S. Donaldson and L. Katzarkov,
\newblock {\it Luttinger surgery along Lagrangian tori and
      non-isotopy for singular symplectic plane curves},
    Math. Ann. {\bf{326}} (2003), 185--203.


    
  \bibitem{BK}
    S. Baldridge and P. Kirk,
    \newblock
        {\it A symplectic manifold homeomorphic but not diffeomorphic to
${\CP}^2\# 3 {\overline   {\CP}}^2$,} Geom. Topol. {\bf{12}} (2008), 919--940.

      \bibitem{DK}
        S. Donaldson and P. Kronheimer,
        \newblock{
          {\it The geometry of four-manifolds,} Oxford Mathematical Monographs,}
        (1990).

\bibitem{DorfLi}
  J. Dorfmeister and T.-J. Li,
  \newblock {\it The minimal genus problem --- a quarter century of progress},
  Acta Math. Sci. Ser. B (Engl. Ed.) {\bf{42}} (2022), 2257--2278.
        
      \bibitem{FPS}
        R. Fintushel, D. Park, and R. Stern, 
\newblock {\it Reverse engineering small 4-manifolds},
Algebr. Geom. Topol. {\bf{7}} (2007), 2103--2116.

\bibitem{FSknots}
        R. Fintushel and R. Stern, 
\newblock {\it Knots, links, and 4-manifolds},
Invent. Math. {\bf{134}} (1998), 363--400.

        
      \bibitem{Furuta}
        M. Furuta,
        \newblock
            {\it Monopole equation and the $\frac{11}{8}$-conjecture},
              Math. Res. Lett. {\bf{8}} (2001), 279--291.
        
   \bibitem{Gao}
 H.    Gao, 
\newblock {\it Representing homology classes of almost definite 4-manifolds},
Topology Appl. {\bf{52}} (1993),  109--120.
     
\bibitem{GS}
R.~Gompf and A.~Stipsicz, 
\newblock {\it {$4$}-manifolds and {K}irby calculus},
\newblock {Graduate Studies in Mathematics} {\bf20} (1999),
{American Mathematical Society, Providence, RI.}


\bibitem{kirby}
  R. Kirby,
  \newblock {\it The topology of 4-manifolds},
  Lecture Notes in Mathematics, {\bf{1374}}.
  Springer-Verlag, Berlin, 1989. 

\bibitem{KrMr}
P. Kronheimer and T. Mrowka, 
\newblock {\it {The genus of embedded surfaces in the projective plane}},
          \newblock {Math. Res. Lett.}, {\bf{1}}, {(1994)}, {797--808}.

\bibitem{Lut}
K. Luttinger, 
\newblock {\it Lagrangian tori in $R^4$},
\newblock {J. Differential Geom.} {\bf{42}} (1995), {220--228}.

          
\bibitem{Mor}
  J.~Morgan,
  \newblock {\it The Seiberg-Witten equations
  and applications to the topology of smooth four-manifolds},
  \newblock Mathematical Notes {\bf44}, Princeton University Press,
  Princeton, NJ, 1996.

\bibitem{MMSz}
  J. Morgan, T. Mrowka and Z. Szab\'o,
  \newblock {\it Product formulas along $T^3$ for Seiberg-Witten invariants},
     Math. Res. Lett. {\bf{4}} (1997),  915--929. 

   \bibitem{Lawson}
T. Lawson,
\newblock {\it The minimal genus problem},
Exposition. Math. {\bf{15}} (1997), {385--431}.
     
\bibitem{LiLi}
  B.-H.   Li and T.-J. Li, 
  \newblock {\it  Minimal genus smooth
    embeddings in $S^2\times S^2$ and ${\CP}^2 \# n{\overline {\CP}}^2$ with
    $n\leq 8$,}
  Topology {\bf{37}} (1998),  575--594.

\bibitem{OSzThom}
  P. Ozsv\'ath and Z. Szab\'o,
  \newblock {\it The symplectic Thom conjecture,}
  \newblock Ann. of Math. (2) {\bf{151}} (2000),  93--124. 

  
  \bibitem{Ruber}
D. Ruberman, 
\newblock {\it The minimal genus of an embedded surface of non-negative square in a rational surface},
\newblock Turkish J. Math. {\bf{20}} (1996), 129--133.

  
\bibitem{taubes}
  C.~Taubes,
  \newblock {\it  The Seiberg-Witten invariants and symplectic forms},
  \newblock Math. Res. Lett. {\bf1} (1994),  809--822.

  \bibitem{T}
  C.~Taubes,
  \newblock {\it The Seiberg-Witten invariants and 4-manifolds with
    essential tori},
  Geom. Topol. {\bf{5}} (2001), 441--519. 

\bibitem{Wall}
  C. T. C.~Wall,
  \newblock {\it On the orthogonal groups of unimodular quadratic forms. II.},
  J. Reine Angew. Math. {\bf{213}} (1963/64), 122--136. 

  
\bibitem{W}
  E.~Witten, 
  \newblock {\it Monopoles and four-manifold},
  \newblock Math. Res. Lett. {\bf1} (1994), 769--796. 

\end{thebibliography}
\end{document}